\def\N{\mathbb N}
\def\Z{\mathbb Z}
\def\R{\mathbb R}
\def\Q{\mathbb Q}
\def\Finmb{\text{Fin}(-\beta)}

\def\Zmb{{\mathbb Z}_{-\beta}}
\def\A{\mathcal A}
\def\ol{\overline}
\def\pfz{\begin{proof}}
\def\pfk{\end{proof}}

\documentclass[11pt]{article}
\usepackage{graphicx, amssymb, amsmath,amsthm}
\newtheorem{thm}{Theorem}[section]
\newtheorem{lem}[thm]{Lemma}

\newtheorem{ex}[thm]{Example}
\newtheorem{de}[thm]{Definition}

\begin{document}

\title{Arithmetics in numeration systems\\ with negative quadratic base}
\author{Z. Mas\'akov\'a\footnote{e-mail: zuzana.masakova@fjfi.cvut.cz}, T. V\'avra\\[1mm]
{\normalsize Department of Mathematics FNSPE, Czech Technical University in Prague}\\
{\normalsize Trojanova 13, 120 00 Praha 2, Czech Republic}}
\date{}

\maketitle

\begin{abstract}
We consider positional numeration system with negative base $-\beta$, as introduced by Ito and Sadahiro.
In particular, we focus on arithmetical properties of such systems when $\beta$ is a quadratic Pisot number.
We study a class of roots $\beta>1$ of polynomials $x^2-mx-n$, $m\geq n\geq 1$, and show that
in this case the set ${\rm Fin}(-\beta)$ of finite $(-\beta)$-expansions is closed under addition, although it is not closed under
subtraction. A particular example is $\beta=\tau=\frac12(1+\sqrt5)$, the golden ratio. For such $\beta$, we determine the
exact bound on the number of fractional digits appearing in arithmetical operations. We also show that the set of $(-\tau)$-integers coincides on the
positive half-line with the set of $(\tau^2)$-integers.
\end{abstract}

%%%%%%%%%%%%%%%%%%%%%%%%%%%%%%%%%%%%%%%%%%%%%%%%%%%%%%%%%%%%%%%%%%%%%%%%%%
%%%%%%%%%%%%%%%%%%%%%%%%%%%%%%%%%%%%%%%%%%%%%%%%%%%%%%%%%%%%%%%%%%%%%%%%%%
%%%%%%%%%%%%%%%%%%%%%%%%%%%%%%%%%%%%%%%%%%%%%%%%%%%%%%%%%%%%%%%%%%%%%%%%%%
%%%%%%%%%%%%%%%%%%%%%%%%%%%%%%%%%%%%%%%%%%%%%%%%%%%%%%%%%%%%%%%%%%%%%%%%%%
\section{Introduction}

In practically all fields of applied sciences one meets problems requiring efficient computational methods. An indispensable key for
developing such methods is to have fast algorithms for performing arithmetical operations with high precision. The first of the two
aspects --- speed --- can be reached for example using parallelization of algorithms for addition and multiplication. However, it has been
shown~\cite{mazenc} that this can only be achieved allowing redundancy in number representation.
The second aspect --- accuracy --- calls for special treatment
of different classes of irrational numbers by using e.g. exact arithmetics in algebraic number fields.
All this motivates the study of non-standard number systems.

Usually, one represents numbers in the standard positional number system with base 10 or base 2, (the so-called decimal or binary representation of numbers).
Changing the base for any integer $b\geq 2$ does not bring much new. In 1957, R\'enyi~\cite{Renyi} introduced the possibility of representing numbers in a system with non-integer base $\beta>1$. For every non-negative real number $x$, we have the $\beta$-expansion of $x$ of the form
$$%\begin{equation}%\label{eq:1}
x=\sum_{i=-\infty}^{k} x_i\beta^i\,,\qquad x_i\in\{0,1,\dots,\lceil\beta\rceil-1\}\,,
$$%\end{equation}
where the digits $x_i$ are obtained by the greedy algorithm. In analogy with standard numeration with integer base, we define the set $\Z_\beta$ of $\beta$-integers, which have vanishing digits at negative powers of $\beta$. We also define the set ${\rm Fin}(\beta)$ of numbers with finite $\beta$-expansion,
i.e.\ numbers whose  $\beta$-expansion has only finitely many non-zero digits.

Many new interesting phenomena appear when considering $\beta$-expansions for non-integer base $\beta$. For example, whereas in base $b\in\N$, every finite string of non-negative integers $< b$ is admissible as the greedy expansion of some $x$, in base $\beta\notin\N$, this is no longer true. The conditions  for admissibility of digit strings for general base $\beta$ has been given by Parry~\cite{Parry} using the lexicographical ordering of digit strings.

But probably the most remarkable novelty is that $\Z_\beta$ is no longer equal to the set $\Z$ of rational integers; its elements are not equidistant on the real line and $\Z_\beta$ is not a ring (i.e.\ closed under addition and multiplication), as it is the case for $\Z$. Even more strange, addition of $\beta$-integers may result in an infinite $\beta$-expansion. Such properties were studied by many authors. Among the most important results on arithmetics with $\beta$-expansions is Schmidt's description of bases $\beta$ for which rational numbers have periodic $\beta$-expansions~\cite{schmidt}, or the necessary condition on $\beta$, so that ${\rm Fin(\beta)}$ is a ring, given by Frougny and Solomyak~\cite{FruSo}. Others have studied the fractional part appearing in arithmetical operations with $\beta$-expansions, see e.g.~\cite{BuFrGaKr,GuMaPeAA,Bernat}. On-line computability of arithmetic operations was studied in~\cite{Fru,FruSu}.
Let us note that many questions about arithmetics in the numeration systems with positive real base $\beta$ remain open.

%\smallskip
Recently, Ito and Sadahiro~\cite{ItoSadahiro} suggested to study positional systems with negative base $-\beta$, where $\beta>1$. Here one obtains a representation of every (both positive or negative) real number in the form
$$%\begin{equation}\label{eq:2}
x=\sum_{i=-\infty}^{k} x_i(-\beta)^i\,,\qquad x_i\in\{0,1,\dots,\lfloor\beta\rfloor\}\,.
$$%\end{equation}
Ito and Sadahiro have provided a condition for admissibility of digit strings as $(-\beta)$-expansions and shown some properties of the dynamical system connected to $(-\beta)$-numeration. Their work on dynamical aspects has been continued in~\cite{ChiaraFrougny}.
The authors of~\cite{ADMP} define the set $\Z_{-\beta}$ of $(-\beta)$-integers and focus on its geometrical features. Some arithmetical properties of $(-\beta)$-numeration systems are studied in~\cite{MaPeVa}. Among other, the validity of a conjecture of Ito and Sadahiro is established, which states that
if $\beta>1$ is the root of $x^2-mx+n$, $m,n\in\N$, $m\geq n+2\geq 3$, then the set of ${\rm Fin}(-\beta)$ of finite $(-\beta)$-expansions is a ring. One also provides bounds on the length of the fractional part arising by adding and multiplying $(-\beta)$-integers, this in case that $\beta$ is a root of $x^2-mx-1$ for $m\geq 2$, and that $\beta$ is a root of $x^2-mx+1$ for $m\geq 3$.

In the present paper we complete the arithmetical study for quadratic negative bases started in~\cite{MaPeVa}.
In particular, we focus on roots $\beta>1$ of polynomials $x^2-mx-n$, $m\geq n\geq 1$, and show that
in this case the set ${\rm Fin}(-\beta)$ of finite $(-\beta)$-expansions is closed under addition, although it is not closed under
subtraction. We also provide exact bound on the number of fractional digits appearing in arithmetical operations for the golden ratio $\tau=\frac12(1+\sqrt5)$,
which is a case missing in the study~\cite{MaPeVa}.
We also prove a curious coincidence between $(-\tau)$-integers and $(\tau^2)$-integers. For that, we need to describe the distances between consecutive
$(-\tau)$-integers and morphism under which the infinite word coding their ordering is invariant.

%%%%%%%%%%%%%%%%%%%%%%%%%%%%%%%%%%%%%%%%%%%%%%%%%%%%%%%%%%%%%%%%%%%%%%%%%%
\section{Positive base number system}

Let $\beta>1$. The R\'enyi $\beta$-expansion of a real number $x\in[0,1)$ can be found as a coding of the orbit of the point $x$
under the transformation $T_\beta:[0,1)\mapsto[0,1)$, given by the prescription
$T_\beta(x):=\beta x-\lfloor\beta x\rfloor$. Every $x\in[0,1)$ is a sum of the infinite series
\begin{equation}\label{eq:1}
x=\sum_{i=1}^\infty\frac{x_i}{\beta^i}\,,\qquad\hbox{where }\ x_i=\lfloor\beta T_\beta^{i-1}(x)\rfloor\quad\hbox{for }i=1,2,3,\dots
\end{equation}
Directly from the definition of the transformation $T_\beta$ we can derive that the digits $x_i$ take values in the set $\{0,1,2,\cdots,\lceil\beta\rceil -1\}$ for $i=1,2,3,\cdots$.

\begin{de}\label{de:betaexpansion}
The expression of $x$ in the form~\eqref{eq:1} is called the $\beta$-expansion of $x\in[0,1)$. The number $x$ is thus represented by the infinite word
$$
d_\beta(x)=x_1x_2x_3\cdots\in{\mathcal A}^\N
$$
over the alphabet $\A=\{0,1,2,\dots,\lceil\beta\rceil-1\}$.
\end{de}

From the definition of the transformation $\beta$ we can derive another important property, namely that the ordering on real numbers is carried over to
the ordering of $\beta$-expansions. In particular, we have for $x,y\in[0,1)$ that
$$
x\leq y \quad\iff\quad d_\beta(x) \preceq d_\beta(y)\,,
$$
where $\preceq$ is the lexicographical order on $\A^{\N}$, (ordering on the alphabet $\A$ is usual, $0<1<2<\cdots<\lceil\beta\rceil-1$).

In~\cite{Parry}, Parry has provided a criterion which decides whether an infinite word in $\A^\N$ is or is not the $\beta$-expansion of some real number $x$.
The criterion is formulated using the so-called infinite expansion of $1$, denoted by $d^*_\beta(1)$, defined as a limit in the space $\A^\N$ equipped with the product topology,
$$
d_\beta^*(1):=\lim_{\varepsilon\to0+}d_\beta(1-\varepsilon)\,.
$$
According to Parry, the string $x_1x_2x_3\cdots \in\A^\N$ represents the $\beta$-expansion of a number $x\in[0,1)$ if and only if
\begin{equation}\label{eq:Parry}
x_ix_{i+1}x_{i+2}\cdots \prec d_\beta^*(1)\quad \hbox{ for every }\ i=1,2,3,\cdots\,.
\end{equation}

%A $\beta$-representation of a real number $x$ is an expression of the form
%$$
%x=\sum_{i=1}^\infty\frac{x_i}{\beta^i}\,,\qquad\hbox{where $x_i\in\Z$.}
%$$
%It is usually denoted by
%$$
%x_
%$$
%The largest (in the short-lex order) $\beta$-representation is called the $\beta$-expansion of $x$. found by the greedy algorithm. Sc

The notion of $\beta$-expansion can be naturally extended to all non-negative real numbers:

\begin{de}\label{de:betaexpansionreal}
Let $\beta>1$ and $x\geq 0$. The expression
\begin{equation}
x=x_k\beta^{k}+x_{k-1}\beta^{k-1}+x_{k-2}\beta^{k-2}+\cdots\,,\quad \text{where $k\in\Z$ and $x_i\in\Z$ for $i\leq k$}\,,
\end{equation}
is a $\beta$-representation of $x$. The $\beta$-expansion of $x$ is the particular
$\beta$-representation satisfying  $x_kx_{k-1}x_{k-2}\cdots= d_\beta(y)$ for some $y\in[0,1)$.
\end{de}

Note that Definition~\ref{de:betaexpansionreal} of $\beta$-expansion is in accordance with Definition~\ref{de:betaexpansion}.
The $\beta$-representation of $x$ is sometimes written as the digit string
\begin{equation}
\begin{array}{ll}
x_k\cdots x_0\bullet x_{-1}x_{-2}\cdots\,,\qquad&\text{for $k\geq 0$, or}\\[2mm]
0\bullet 0^{-k-1}x_k x_{k-1}\cdots\,,\qquad&\text{for $k<0$,}
\end{array}
\end{equation}
with the notation $0^{j}$ standing for $j$ digits 0 repeated.
If the $\beta$-representation is in the same time the $\beta$-expansion of $x$, we denote the cooresponding digit string
by $\langle x\rangle_\beta$.

When the digit string $x_kx_{k-1}x_{k-2}\cdots$ ends in infinitely many 0's, we say that
the $\beta$-expansion is finite and omit the ending 0's.
Real numbers $x$ having in their $\beta$-expansion vanishing digits $x_i$ for all $i<0$ are usually called $\beta$-integers and the set of $\beta$-integers
is denoted by $\Z_\beta$,
$$
\Z_\beta=\big\{x\in\R \,\big|\, \langle |x|\rangle_\beta=x_k\cdots x_0\bullet \big\}\,.
$$
The set of numbers with finite $\beta$-expansion is then
$$
{\rm Fin(\beta)}= \bigcup_{k\in\Z} \beta^k\Z_\beta\,.
$$
The set ${\rm Fin(\beta)}$ is in general not closed under addition and multiplication. The description of bases $\beta$, for which ${\rm Fin(\beta)}$ is a ring,
is a very difficult open question. Only partial results are known.

One can study the arithmetics on $\beta$-expansions in more detail: even though addition or multiplication of two $\beta$-integers may result in an infinite $\beta$-expansion, one can define the following quantities,
\begin{equation}\label{eq:LLkladnabaze}
\begin{aligned}
L_\oplus(\beta)&=\min\{l\in\N\mid \forall\,x,y\in\Z_{\beta},\ x+y\in{\rm Fin}(\beta)\Rightarrow x+y\in\beta^{-l}\Z_{\beta}\}\,,\\
L_\otimes(\beta)&=\min\{l\in\N\mid \forall\,x,y\in\Z_{\beta},\ x\cdot y\in{\rm Fin}(\beta)\Rightarrow x\cdot y\in\beta^{-l}\Z_{\beta}\}\,.
\end{aligned}
\end{equation}
describing the maximal length of finite fractional part possibly arising when summing, resp. multiplying $\beta$-integers.

%%%%%%%%%%%%%%%%%%%%%%%%%%%%
\begin{ex}
Consider for the base of the numeration system the golden ratio $\tau=\frac12(1+\sqrt5)$, root of the quadratic polynomial $x^2-x-1$. Every $x\geq 0$ has its $\tau$-expansion of the form
$$
x=\sum_{i=-\infty}^{k} x_i\tau^i\,,\qquad x_i\in\{0,1\}\,.
$$
where according to the Parry condition~\eqref{eq:Parry}, the digit sequence $x_ix_{i-1}x_{i-2}\cdots$ for every $i\leq k$ satisfies
$$
x_ix_{i-1}x_{i-2}\cdots \prec (10)^\omega\,.
$$
Here $(10)^\omega$ denotes infinite repetition of the string $10$. This condition can be reformulated in a more comprehensible way: the digit sequence $x_kx_{k-1}x_{k-2}\cdots$ does not end in an infinite repetition of $10$ and does not contain two consecutive digits 1.
The latter requirement is intuitively obvious, since the greedy algorithm prefers replacing the digit string $011$ by $100$, in accordance with the equation $\tau^{i+2}=\tau^{i+1}+\tau^i$.

The set $\Z_\tau$ of $\tau$-integers can be expressed as
$$
\begin{aligned}
\Z_\tau &= \left\{\pm\sum_{i=0}^k x_i\tau^i \ \biggm|\ x_i\in\{0,1\},\ x_i\cdot x_{i+1}=0\right\} = \\
&=\pm\big\{0,1,\tau,\tau^2,\tau^2+1,\tau^3,\tau^3+1,\tau^3+\tau,\tau^4,\tau^4+1,\dots\big\}\,.
\end{aligned}
$$
Drawn on the real line, we see that the distances between consecutive $\tau$-integers take values $1$ and $\tau^{-1}$, see Figure~\ref{f}.

\begin{figure}[h]
\begin{center}
\setlength{\unitlength}{0.8pt}
\begin{picture}(430,59)
\put(0,23){\line(1,0){418}}
\put(22,30){$\overbrace{\hspace*{44\unitlength}}^{1}$}
\put(70,30){$\overbrace{\hspace*{26\unitlength}}^{1/\tau}$}
\put(100,30){$\overbrace{\hspace*{44\unitlength}}^{1}$}
\put(148,30){$\overbrace{\hspace*{44\unitlength}}^{1}$}
\put(196,30){$\overbrace{\hspace*{26\unitlength}}^{1/\tau}$}
\put(226,30){$\overbrace{\hspace*{44\unitlength}}^{1}$}
\put(274,30){$\overbrace{\hspace*{26\unitlength}}^{1/\tau}$}
\put(304,30){$\overbrace{\hspace*{44\unitlength}}^{1}$}
\put(352,30){$\overbrace{\hspace*{44\unitlength}}^{1}$}
\put(20,21){\line(0,1){4}}
\put(20,5){$0$}
\put(68,21){\line(0,1){4}}
\put(66,5){$1$}
\put(98,21){\line(0,1){4}}
\put(96,5){$\tau$}
\put(146,21){\line(0,1){4}}
\put(143,5){$\tau^2$}
\put(194,21){\line(0,1){4}}
\put(178,5){$\tau^2\!\!+\!1$}
\put(224,21){\line(0,1){4}}
\put(224,5){$\tau^3$}
\put(272,21){\line(0,1){4}}
\put(255,5){$\tau^3\!\!+\!1$}
\put(302,21){\line(0,1){4}}
\put(296,5){$\tau^3\!\!+\!\tau$}
\put(350,21){\line(0,1){4}}
\put(348,5){$\tau^4$}
\put(398,21){\line(0,1){4}}
\put(387,5){$\tau^4\!\!+\!1$}
\end{picture}
\end{center}
\caption{Several smallest non-negative $\tau$-integers drawn on the real line.}\label{f}
\end{figure}
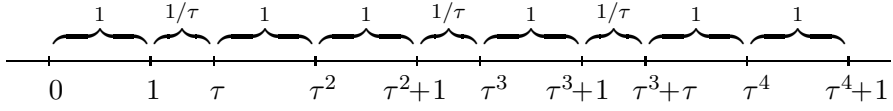

The $\tau$-expansions of the smallest few non-negative $\tau$-integers are given in the following table.
Note that they are increasing in the short-lex order.
$$
\begin{array}{r@{\ }c@{\ }r@{\,,\hspace*{0.8cm}}r@{\ }c@{\ }r}
\langle 0 \rangle_\tau &=& 0\bullet & \langle \tau^3 \rangle_\tau &=& 1000\bullet\,,\\[1mm]
\langle 1 \rangle_\tau &=& 1\bullet & \langle \tau^3+1 \rangle_\tau &=& 1001\bullet\,,\\[1mm]
\langle \tau \rangle_\tau &=& 10\bullet & \langle \tau^3+\tau \rangle_\tau &=& 1010\bullet\,,\\[1mm]
\langle \tau^2 \rangle_\tau &=& 100\bullet & \langle \tau^4 \rangle_\tau &=& 10000\bullet\,,\\[-1mm]
\langle \tau^2+1 \rangle_\tau &=& 101\bullet & &\vdots&
\end{array}
$$

It has been shown in~\cite{BuFrGaKr} that the quantities $L_\oplus(\tau)$, $L_\otimes(\tau)$ defined by~\eqref{eq:LLkladnabaze} take values
$L_\oplus(\tau)=L_\otimes(\tau)=2$. As an example, we can consider
$$
x=y=1 \quad\Rightarrow\quad x+y=1+1=2=\tau+\tau^{-2}=10\bullet01\,.
$$
For multiplication,
$$
x=y=\tau^2+1 \quad\Rightarrow\quad xy=(\tau^2+1)^2=\tau^5+\tau+\tau^{-2}=100010\bullet01\,.
$$
\end{ex}

%%%%%%%%%%%%%%%%%%%%%%%%%%%%

%%%%%%%%%%%%%%%%%%%%%%%%%%%%%%%%%%%%%%%%%%%%%%%%%%%%%%%%%%%%%%%%%%%%%%%%%%
\section{Negative base number system}

In analogy to the R\'enyi expansion of numbers using the transformation of the interval $[0,1)$, Ito and Sadahiro have defined the $(-\beta)$-expansion of numbers using the transformation $T_{-\beta}:[l_\beta,r_\beta)\mapsto[l_\beta,r_\beta)$, where
$$
l_\beta=-\frac{\beta}{\beta+1},\quad r_\beta=1+l_\beta=\frac1{\beta+1}\,.
$$
The transformation $T_{-\beta}$ is defined by
\begin{equation}\label{eq:7}
T_{-\beta}(x):=-\beta x -\lfloor-\beta x - l_\beta\rfloor\,.
\end{equation}
Every real $x\in [l_\beta,r_\beta)$ can be written as
\begin{equation}\label{eq:4}
x=\sum_{i=1}^\infty\frac{x_i}{(-\beta)^i}\,,\qquad\hbox{where }\ x_i=\lfloor-\beta T_{-\beta}^{i-1}(x)-l_\beta\rfloor\quad\hbox{for }i=1,2,3,\dots
\end{equation}

\begin{de}\label{de:-betaexpansion}
We call the expression~\eqref{eq:4} the Ito-Sadahiro $(-\beta)$-expansion of $x\in [l_\beta,r_\beta)$. We denote the corresponding digit string by
$$
d_{-\beta}(x)=x_1x_2x_3\cdots\,.
$$
\end{de}

One can easily show from~\eqref{eq:7} that the digits $x_i$, $i\geq 1$, take values in the set $\A=\{0,1,2,\dots,\lfloor\beta\rfloor\}$.
In this case, the ordering on the set of infinite words over the alphabet $\A$ which would correspond to the ordering of real numbers is the so-called alternate ordering: We say that
$$
x_1x_2x_3\cdots \prec_{\hbox{\tiny alt}} y_1y_2y_3\cdots
$$
if for the minimal index $j$ such that $x_j\neq y_j$ it holds that $x_j(-1)^j<y_j(-1)^j$.
In this notation, we can write for arbitrary $x,y\in [l_\beta,r_\beta)$ that
$$
x\leq y \quad\iff\quad d_{-\beta}(x) \preceq_{\hbox{\tiny alt}} d_{-\beta}(y)\,.
$$

In their paper, Ito and Sadahiro have provided a criterion to decide whether an infinite word over $\A^\N$ is admissible as $d_{-\beta}(x)$ for some $x\in [l_\beta,r_\beta)$. The criterion is given using two infinite words, namely
$$
d_{-\beta}(l_\beta)\quad\hbox{ and }\quad d^*_{-\beta}(r_\beta):=\lim_{\varepsilon\to0+} d_{-\beta}(r_\beta-\varepsilon)\,.
$$
These two infinite words have close relation: If $d_{-\beta}(l_\beta)$ is purely periodic with odd period length, i.e.
$d_{-\beta}(l_\beta)=(d_1d_2\cdots d_{2k+1})^\omega$, then $d_{-\beta}^*(r_\beta)=\big(0d_1d_2\cdots (d_{2k+1}-1)\big)^\omega$.
(As usual, the notation $v^\omega$ stands for infinite repetition of the string $v$.)
In all other cases one has $d_{-\beta}^*(r_\beta)=0d_{-\beta}(l_\beta)$.

Ito and Sadahiro have shown that an infinite word $x_1x_2x_3\cdots$ represents $d_{-\beta}(x)$ for some $x\in[l_\beta,r_\beta)$ if and only if
for every $i\geq 1$ it holds that
\begin{equation}\label{eq:admissIS}
d_{-\beta}(l_\beta) \preceq_{\hbox{\tiny alt}} x_ix_{i+1}x_{i+2}\cdots \prec_{\hbox{\tiny alt}} d_{-\beta}^*(r_\beta)\,.
\end{equation}

We now provide the definition of $(-\beta)$-expansions of every real number $x$. (Note that in the negative base number system we can represent
negative numbers using non-negative digits without need of sign.)

In~\cite{ItoSadahiro} it is suggested to find the expansion of a number $x\notin [l_\beta,r_\beta)$ by dividing it by a suitable power of $(-\beta)$
so that $y:=(-\beta)^{-k}x\in[l_\beta,r_\beta)$, finding the expansion of $y$ and multiplying it back by $(-\beta)^k$. The expression for $x$ provided by such procedure, however, depends on chosen $k$, so the prescription must be modified, in order to give a unique $(-\beta)$-expansion for every real $x$.

\begin{de}\label{de:-betaexpansionreal}
Let $\beta>1$ and $x\in\R$. The expression
\begin{equation}
x=x_k(-\beta)^{k}+x_{k-1}(-\beta)^{k-1}+x_{k-2}(-\beta)^{k-2}+\cdots\,,\quad \text{where $k\in\Z$ and $x_i\in\Z$ for $i\leq k$}\,,
\end{equation}
is a $(-\beta)$-representation of $x$. The $(-\beta)$-expansion of $x$ is the particular
$(-\beta)$-representation satisfying  $x_kx_{k-1}x_{k-2}\cdots= d_{-\beta}(y)$ for some $y\in(l_\beta,r_\beta)$.
\end{de}

Again, we write the $(-\beta)$-representations using the corresponding digit strings where the symbol $\bullet$ stands for fractional point
separating the digits at non-negative and negative powers of the base,
$$
\begin{array}{ll}
x_k\cdots x_0\bullet x_{-1}x_{-2}\cdots\,,\qquad&\text{for $k\geq 0$, or}\\[2mm]
0\bullet 0^{-k-1}x_k x_{k-1}\cdots\,,\qquad&\text{for $k<0$.}
\end{array}
$$
If the digit string corresponds to the $(-\beta)$-expansion of $x$, we denote it by $\langle x\rangle_{-\beta}$.

Note that for the sake of uniqueness, in Definition~\ref{de:-betaexpansionreal} we do not allow  a
$(-\beta)$-expansion to be given by $d_{-\beta}(l_\beta)$. This is
because of the following undesirable phenomena: Denoting
$d_{-\beta}(l_\beta)=d_1d_2d_3\cdots$, we have the equality
$$
\begin{aligned}
(-\beta)&+d_1(-\beta)^0+\frac{d_2}{(-\beta)}+\frac{d_3}{(-\beta)^2}+\frac{d_4}{(-\beta)^3}+\cdots = \\
&= \frac{d_1}{(-\beta)}+\frac{d_2}{(-\beta)^2}+\frac{d_3}{(-\beta)^3}+\frac{d_4}{(-\beta)^4}+\cdots
\end{aligned}
$$
where both $1d_1d_2d_3\cdots$ and $d_1d_2d_3\cdots$ are admissible digit strings. But whereas $01d_1d_2d_3\cdots$
is also admissible, $0d_1d_2d_3\cdots$ is not, and so we prefer to define the $(-\beta)$-expansion of $l_\beta$
as
$$
\langle l_\beta\rangle_{-\beta} = 1d_1\bullet d_2d_3\cdots\,.
$$

Similarly as in the case of positive base numeration, we define the $(-\beta)$-integers, forming the set
$$
\Z_{-\beta}=\big\{x\in\R \,\big|\, \langle x\rangle_{-\beta}=x_k\cdots x_0\bullet \big\}\,.
$$
The set of numbers with finite $(-\beta)$-expansion is defined by
$$
{\rm Fin(-\beta)}= \bigcup_{k\in\Z} (-\beta)^k\Z_{-\beta}\,.
$$

We also define the quantities describing maximal length of fractional part arising when summing or multiplying $(-\beta)$-integers,
\begin{equation}\label{eq:LLzapornabaze}
\begin{aligned}
L_\oplus(-\beta)&=\min\{l\in\N\mid \forall\,x,y\in\Z_{-\beta},\ x+y\in{\rm Fin}(-\beta)\Rightarrow x+y\in(-\beta)^{-l}\Z_{-\beta}\}\,,\\
L_\otimes(-\beta)&=\min\{l\in\N\mid \forall\,x,y\in\Z_{-\beta},\ x\cdot y\in{\rm Fin}(-\beta)\Rightarrow x\cdot y\in(-\beta)^{-l}\Z_{-\beta}\}\,.
\end{aligned}
\end{equation}

%%%%%%%%%%%%%%%%%%%%%%%%%%%%%%%%%%%%%%%%%%%%%%%%%%%%%%%%%%%%%%%%%%%%%%%%%%
\section{Pisot numbers}

Since the present study focuses on a special class of algebraic numbers, let us recall some number-theoretical notions needed.
A complex number $\beta$ is called an algebraic number, if it is a root of a monic polynomial $f(x)=x^n + a_{n-1}x^{n-1}+\cdots +a_1x + a_0$, with rational coefficients $a_0,\dots,a_{n-1}\in\Q$. If $f$ has the minimal degree among all polynomials satisfying $g\in\Q[x]$, $g(\beta)=0$, then
 it is called the minimal polynomial of $\beta$ and the degree of $f$ is called the degree of $\beta$. The other roots of the minimal polynomial are the algebraic conjugates of $\beta$.

If the minimal polynomial of $\beta$ has integer coefficients, $\beta$ is called an algebraic integer.
An algebraic integer $\beta>1$ is called a Pisot number, if all its conjugates are in modulus strictly smaller than $1$.

Let $\beta$ be an algebraic number of degree $r$. The minimal subfield of the field of complex numbers containing $\beta$ is denoted by $\Q(\beta)$
and is of the form
$$
\Q(\beta) = \{c_0+c_1\beta + \cdots + c_{r-1}\beta^{r-1} \mid c_i\in\Q\}\,.
$$
If $\gamma$ is a conjugate of an algebraic number $\beta$, then the fields $\Q(\beta)$ and $\Q(\gamma)$ are isomorphic. The corresponding isomorphism
$\sigma:\Q(\beta)\mapsto\Q(\gamma)$ is given by the prescription
\begin{equation}\label{eq:isomorphism}
\sigma:\quad c_0+c_1\beta + \cdots + c_{r-1}\beta^{r-1} \quad\mapsto\quad c_0+c_1\gamma + \cdots + c_{r-1}\gamma^{r-1}\,.
\end{equation}
In particular, it means that $\beta$ is a root of some polynomial $f$ with rational coefficients if and only if $\gamma$ is a root of the same polynomial $f$.

In the field of numeration systems with non-integer bases one often meets a special class of algebraic numbers, namely Pisot numbers. A nice result (see~\cite{schmidt}) is
that Pisot numbers have eventually periodic infinite $\beta$-expansion of 1 (cf.\ Parry condition~\eqref{eq:Parry}). Similarly, in~\cite{ChiaraFrougny} it
is shown that Pisot numbers have also eventually periodic $d_{-\beta}(l_\beta)$, which is useful in the admissibility condition of $(-\beta)$-expansions (cf.~\eqref{eq:admissIS}).

It is known~\cite{Bassino} that among the quadratic numbers, the only ones with eventually periodic infinite $\beta$-expansion of 1 are quadratic Pisot numbers. Similarly,
quadratic Pisot numbers are the only quadratic numbers with eventually periodic $d_{-\beta}(l_\beta)$, see~\cite{MaPeVa}.
It is easy to show that quadratic Pisot numbers are precisely the larger roots of polynomials
$$
\begin{array}{ll}
x^2-mx-n\,,\quad &m,n\in\N,\ m\geq n\geq 1,\\[2mm]
x^2-mx+n\,,\quad &m,n\in\N,\ m\geq n+2\geq 3.
\end{array}
$$
The corresponding  infinite $\beta$-expansions of 1 are
$$
d_\beta^*(1) = \big(m(n-1)\big)^\omega,\quad d_\beta^*(1) = (m-1)(m-n-1)^\omega,\quad \text{respectively.}
$$
For the $(-\beta)$-expansion of $l_\beta$, one obtains
$$
d_{-\beta}(l_\beta) = m(m-n)^\omega,\quad d_{-\beta}(l_\beta) = \big((m-1)n\big)^\omega,\quad \text{respectively.}
$$

%%%%%%%%%%%%%%%%%%%%%%%%%%%%%%%%%%%%%%%%%%%%%%%%%%%%%%%%%%%%%%%%%%%%%%%%%%
\section{Arithmetics in systems with quadratic negative base}

Let us now consider the arithmetical properties of the number
system with base $-\beta$, where $\beta$ belongs to the class of
quadratic Pisot numbers with minimal polynomial $x^2-mx-n$, $m\geq
n\geq 1$. The condition~\eqref{eq:admissIS} of admissibility of
digit strings is now stated using
\begin{equation}\label{eq:admissquadr}
d_{-\beta}(l_\beta)=m(m-n)^\omega \quad\text{ and }\quad
d_{-\beta}^*(r_\beta)= 0m(m-n)^\omega\,.
\end{equation}
As mentioned already in~\cite{MaPeVa}, the set of finite
$(-\beta)$-expansions is not a ring in this case, since for $x=0$,
$y=1$, we have $x,y\in{\rm Fin}(-\beta)$, but $x-y=-1\notin{\rm
Fin}(-\beta)$. For, the $(-\beta)$-expansion of the number $-1$ is
equal to
$$
\langle -1\rangle_{-\beta} = 1m\bullet(m-n+1)^\omega\,.
$$
Nevertheless, this fact does not prevent ${\rm Fin}(-\beta)$ to be
closed under addition, as we shall prove here
(Theorem~\ref{t:uzavrenost}). For that, let us first describe the
set of finite expansions using some combinatorial property. The
following statement can be easily derived from the admissibility
condition~\eqref{eq:admissIS} using~\eqref{eq:admissquadr}.

\begin{lem}\label{l:zakazane}
Let $\beta>1$ be root of $x^2-mx-n$, $m\geq n\geq 1$. Let
$x_i\in\{0,1,\dots,m\}$, for $i\leq N$, where only finitely many $x_i$ are non-zero.

If $m>n$, then $x=\sum_{i=-\infty}^Nx_i(-\beta)^i$ is the
$(-\beta)$-expansion of $x$ if and only if $x_Nx_{N-1}\cdots$
does not contain strings
\begin{align}
m\,(m-n)^{2k}C\,,\quad C\leq m-n-1\,,\ k\in\N_0\,,\label{zakazane:plusn1}\\
m\,(m-n)^{2k+1}D\,,\quad D\geq m-n+1\,,\ k\in\N_0\,.\label{zakazane:plusn2}
\end{align}

If $m=n$, then $x=\sum_{i=-\infty}^Nx_i(-\beta)^i$ is the
$(-\beta)$-expansion of $x$ if and only if $x_Nx_{N-1}\cdots$
does not contain the string
\begin{align}\label{zakazane:plusn3}
&m\,0^{2k+1}D\,,\quad D\geq 1\,,\ k\in\N_0\,,
\end{align}
and it does not end with the string $0\,m\,0^\omega$.
\end{lem}

The following lemma is crucial. It shows that although a finite
$(-\beta)$-representation of a real number $x$ over the alphabet
$\{0,1,\dots,m\}$ contains forbidden strings listed in
Lemma~\ref{l:zakazane}, the $(-\beta)$-expansion of the
represented number is also finite, and, in most cases does not
contain smaller powers of $(-\beta)$.

\begin{lem}\label{lem:pisotminusn}
Let $\beta>1$ be root of $x^2-mx-n$, $m\geq n\geq 1$. Then
$$
x:=\sum_{i=0}^{N}a_i(-\beta)^i\in{\rm Fin}(-\beta)
$$
for arbitrary $a_i\in\{0,1,\dots,m\}$. Moreover, $x\in\Zmb$ except
when both $m>n$ and $a_0= m$, in which case $x\in \frac{1}{-\beta}\Zmb$.
\end{lem}

\begin{proof}
Consider the $(-\beta)$-representation $a_Na_{N-1}\cdots
a_0\bullet$ of $x$. If it is not the $(-\beta)$-expansion of $x$,
then $a_Na_{N-1}\cdots
a_0 0^\omega$ contains one of forbidden strings listed in Lemma~\ref{l:zakazane}. We shall rewrite the
left-most forbidden string in  $a_Na_{N-1}\cdots a_00^\omega$ by
adding a suitable $(-\beta)$-representation of $0$. The new
$(-\beta)$-representation of $x$ is `better' than
$a_Na_{N-1}\cdots a_00^\omega$ in the way that the left-most
forbidden string starts at a lower power of $(-\beta)$. Such rewriting does
not add non-zero digits to the right, (unless we deal with the last occurring forbidden string).
Therefore, by repeating such rewriting rules,
we finish in finitely many steps with a $(-\beta)$-representation which does not contain
any forbidden strings, i.e.\ it is the $(-\beta)$-expansion of $x$.

Since $\beta$ is a root of $x^2-mx-n$, we have
\begin{equation}\label{eq:reprenuly}
1\ m\ \ol{n}\ \bullet=\ol{1}\ \ol{m}\ n\ \bullet=0\,.
\end{equation}
(Here for a digit $d$ we write $\ol{d}$ instead of $-d$.)

We distinguish several cases, according to the type of the left-most forbidden string (cf.\ Lemma~\ref{l:zakazane}).

\noindent{\bf Case 1.}
Consider first that $m>n$ and take the forbidden string (\ref{zakazane:plusn1}), together with two digits $A,B$ in the $(-\beta)$-representation of $x$
at higher powers of $(-\beta)$,
\begin{equation}\label{eq:zak1}
\cdots\ A\ B\ m\ (m-n)^{2k}\ C\cdots \qquad k\in\N_0,\ C\leq m-n-1\,.
\end{equation}
The way to rewrite the forbidden string depends on the digits $A, B$.

\noindent{\bf Subcase 1.1.}
Let $B=0$, and consequently $A\in\{0,1,\dots,m-1\}$. (Otherwise $A0m$ is also forbidden, which contradicts the fact
that we take the left-most forbidden string.) We rewrite
$$
\begin{array}{@{\cdots\quad}ccccc@{\quad\cdots}}
  A & 0 & m & (m-n)^{2k} & C \\
  A+1 & m & (m-n) & (m-n)^{2k} & C
\end{array}
$$
It is easy to verify that now no forbidden string occurs left from the digit $C$, which was our aim.

\noindent{\bf Subcase 1.2.}
Let in~\eqref{eq:zak1} be $B\neq 0$ and $k\geq 1$. Then
$$
\begin{array}{@{\cdots\quad}cccccc@{\quad\cdots}}
A&B & m & (m-n) &(m-n)^{2k-1} & C\\
A& B-1 & 0 & m & (m-n)^{2k-1} & C
\end{array}
$$
Again, the latter may contain a forbidden string only starting from the digit $C$.

\noindent{\bf Subcase 1.3.}
Let $B\neq 0$ and  $k=0$. We write
$$
\begin{array}{@{\cdots\quad}cccc@{\quad\cdots}}
A&B & m & C\\
A& B-1 & 0 & C+n
\end{array}
$$
where the latter has no forbidden strings up to the digit $C+n$.

\noindent{\bf Case 2.}
Take the forbidden string (\ref{zakazane:plusn2}) which occurs for both $m>n$ and $m=n$,
\begin{equation}\label{eq:zak2}
\cdots\ A\ B\ m\ (m-n)^{2k+1}\ D\cdots \qquad k\in\N_0,\ D\geq m-n+1\,.
\end{equation}
The rewriting is analogous to subcases 1.1. and 1.2., subcase 1.3. now has no analogue.

\noindent{\bf Subcase 2.1.}
Let $B=0$, and consequently $A\in\{0,1,\dots,m-1\}$. We rewrite
$$
\begin{array}{@{\cdots\quad}ccccc@{\quad\cdots}}
  A & 0 & m & (m-n)^{2k+1} & D \\
  A+1 & m & (m-n) & (m-n)^{2k+1} & D
\end{array}
$$
where the latter has no forbidden strings up to the digit $D$.

\noindent{\bf Subcase 2.2.}
Let in~\eqref{eq:zak2} be $B\neq 0$. Then
$$
\begin{array}{@{\cdots\quad}cccccc@{\quad\cdots}}
A&B & m & (m-n) &(m-n)^{2k} & D\\
A& B-1 & 0 & m & (m-n)^{2k} & D
\end{array}
$$
where the latter has no forbidden strings up to the digit $D$.

\noindent{\bf Case 3.} Consider $m=n$. According to Lemma~\ref{l:zakazane} it remains to solve the
case that the only forbidden string in the $(-\beta)$-representation of $x$ is $0m$ at the end. Necessarily,
the  $(-\beta)$-representation ends with $A0m$, where $A\leq m-1$. We rewite
$$
\begin{array}{@{\cdots\quad}ccc}
A& 0 & m \\
A+1& m & 0
\end{array}
$$

By that, we have shown that $x\in\Finmb$. In order to show
$x\in\Zmb$, note that in all cases except subcase 1.3, the
rewriting of the forbidden string did not influence the digits
starting from $C$ (resp. D) to the right. Thus, if the original
$(-\beta)$-representation of $x$ had vanishing digits at negative
powers of $(-\beta)$, then the same is valid for the rewritten
$(-\beta)$-representation of $x$. The only case where new non-zero
digits at negative powers of $(-\beta)$ may arise, is 1.3 for
$m>n$, and that only if $x=a_Na_{N-1}\cdots a_0\bullet = \cdots
ABm\bullet$, i.e. $a_0=m$.
\end{proof}

\begin{thm}\label{t:uzavrenost}
Let $\beta>1$ be root of $x^2-mx-n$, $m\geq n\geq 1$. Then
$\Finmb$ is closed under addition, i.e. $x+y\in\Finmb$ for
$x,y\in\Finmb$.
\end{thm}

\pfz
Since addition of $y\in\Finmb$ to an $x\in\Finmb$ can be
decomposed into addition to $x$ of several digits $1$ at different
positions, it is obvious that it suffices to verify the following
implication,
$$
x\in\Finmb\quad\Rightarrow\quad x+1\in\Finmb\,.
$$
In fact, in context of Lemma~\ref{lem:pisotminusn}, we only need
to obtain a finite $(-\beta)$-representation of $x+1$ over the
alphabet $\{0,1,\dots,m\}$. For that we only consider the case
that the $(-\beta)$-expansion of $x$ has the digit $m$ at position
$(-\beta)^0$, which leads to the digit $m+1$ in the
$(-\beta)$-representation of $x+1$. Again, for elimination of the
forbidden digit $m+1$ we use addition of a suitable representation
of $0$.

\noindent{\bf Case 1.} If the $(-\beta)$-expansion of $x$ is of
the form $\cdots A 0 m \bullet \cdots$, then necessarily
$A\in\{0,1,\dots m-1\}$, and we use
$$
\begin{array}{c@{\ \ }c@{\ \ }c@{\ }c@{\ }c@{\ }c@{\ }c@{\ }c}
x+1&=&\cdots& A & 0 &(m+1)&\bullet&\cdots\\
0  &=&      & 1 & m & \ol{n}&\bullet& \\
\hline \rule{0pt}{12pt}
 x+1&=&\cdots&A+1& m & (m-n+1)&\bullet&\cdots
\end{array}
$$
The latter is over the alphabet $\{0,1,\dots,m\}$.

\noindent{\bf Case 2.} Consider the $(-\beta)$-expansion of $x$ of
the form $\cdots B m \bullet (m-n)\cdots$, with
$B\in\{1,2,\dots,m\}$. Then
$$
\begin{array}{c@{\ \ }c@{\ \ }c@{\ }c@{\ }c@{\ }c@{\ }c@{\ }c}
x+1&=&\cdots& B &(m+1)&\bullet& (m-n)&\cdots\\
0  &=&  & \ol{1} &\ol{m} &\bullet& n& \\
\hline\rule{0pt}{12pt}
 x+1&=&\cdots& B-1& 1 &\bullet& m&\cdots
\end{array}
$$
The latter is over the alphabet $\{0,1,\dots,m\}$.

\noindent{\bf Case 3.} According to Lemma~\ref{l:zakazane}, it
remains to consider the $(-\beta)$-expansion of $x$ of the form
$\cdots B m \bullet X_1\cdots X_{k} Y\cdots$, where $B\in\{1,2,\dots,m\}$, $k\geq 1$,
$X_i\in\{m\!-\!n\!+\!1,\dots ,m\}$ and $Y\in\{0,1,\dots ,m-n\}$.
Here, we shall use a more complicated $(-\beta)$-representation of 0, which we obtain
by repeated use of~\eqref{eq:reprenuly}, namely
$$
0=1\ (m+1)\ (m-n+1)\ \cdots\ (m-n+1)\ (m-n)\ \ol{n}\,.
$$
Then
$$
\begin{array}{c@{\ }c@{\ }c@{\ }c@{\ }c@{\ \,}c@{\ \,}c@{\ }c@{\ }c@{\ \,}c@{\ }c@{\,}c}
x+1&=& \cdots & B & (m+1)&\bullet&X_1 &\ldots & X_{k-1} & X_k & Y&\cdots\\
0  &=&        &\ol{1} & \ol{m+1}&\bullet & \ol {m\!-\!n\!+\!1}&\ldots&\ol {m\!-\!n\!+\!1}&\ol{m-n}&n&
\\
\hline
\rule{0pt}{12pt} x+1&=&\cdots & B-1 & 0&\bullet & \tilde{X}_1 &\ldots & \tilde{X}_{k-1}& \tilde{X}_k& Y\!+\!n&\cdots\\
\end{array}
$$
where $\tilde{X}_i=X_i\!-\!(m\!-\!n\!+\!1)$ for $i=1,\dots,k-1$ and $\tilde{X}_k=X_k\!-\!(m\!-\!n)$.
The latter representation of $x+1$ is over the alphabet $\{0,1,\dots,m\}$ and therefore by Lemma~\ref{lem:pisotminusn},
$x+1\in\Finmb$.
\pfk

%$$
%\begin{array}{cccccccccccc}
%x+1&=& \cdots & B & (m+1)&\bullet&X_1 &\ldots & X_{k-1} & X_k & Y&\quad\cdots\\
%0  &=&        &\ol{1} & \ol{m+1}&\bullet & \ol {m\!-\!n\!+\!1}&\ldots&\ol {m\!-\!n\!+\!1}&\ol{m-n}&n& \\
%\hline
%x+1&=&\cdots & B-1 & 0&\bullet & X_1\!-\!(m\!-\!n\!+\!1) &\ldots & X_{k-1}\!-\!(m\!-\!n\!+\!1)& x_k\!-\!(m\!-\!n)& Y\!+\!n&\quad\cdots\\
%\end{array}
%$$

%%%%%%%%%%%%%%%%%%%%%%%%%%%%%%%%%%%%%%%%%%%%%%%%%%%%%%%%%%%%%%%%%%%%%%%%%%
\section{Bound on length of fractional part}

In this section we focus on the quantities $L_\oplus(-\beta)$, $L_\otimes(-\beta)$. They have been studied for quadratic Pisot units already in~\cite{MaPeVa}.
However, the method used there does not allow us to obtain exact value for $\beta$ equal to the golden ratio $\tau=\frac12(1+\sqrt5)$. We therefore determine the value $L_\oplus(-\tau)$, $L_\otimes(-\tau)$ here, see Theorem~\ref{thm:tau}. The proof uses a strong relation of $\tau$- and $(-\tau)$-representations. In order to distinguish the two in notation, we shall write $\bullet_\tau$, $\bullet_{-\tau}$ for the fractional point separating digits at non-negative and negative powers of the base in the
two numeration systems.

We first show a lemma putting into relation $\tau$- and $(-\tau)$-representations. Note that statement in item 1.\ of Lemma~\ref{celacisla}
is known. We nevertheless include its proof in order to keep the paper self-contained.

\begin{lem}\label{celacisla}
Let $a_i\in\{0,1\}$, for $i=0,\dots,N$. Then
\begin{enumerate}
\item $x=\sum_{i=0}^N a_i\tau^i\in\Z_\tau\,,$
\item $y=\sum_{i=0}^N a_i(-\tau)^i\in\Z_{-\tau}\,.$
\item If $a_i \cdot a_{i-1}=0$ for all $i\in\{1,\dots,N\}$,  then $y\in(-\tau)\Z_{-\tau}\,.$
\end{enumerate}
\end{lem}

\begin{proof}
Item 1.\ is shown as follows: If $a_n\cdots a_0\bullet_\tau$ is not a $\tau$-expansion,
then it contains two consecutive digits 1. If $i$ is the greatest index $\leq N$ such that
$a_i=a_{i-1}=1$, then $a_{i+1}=0$ and we can rewrite
$$
a_{i+1}\tau^{i+1}+a_i\tau^i+a_{i-1}\tau^{i-1}=\tau^{i+1}\,.
$$
We obtain a new $\tau$-representation of $x$, where the sequence
$100$ has been replaced by $011$. The digit sum has strictly
decreased. It is obvious that we can continue to obtain, in
finitely many steps, the $\tau$-expansion of $x$, i.e. a
$\tau$-representation of $x$ not containing the forbidden string
11.

The proof of item 2.\ follows directly from
Lemma~\ref{lem:pisotminusn}. In item 3., if $a_0=0$, the proof
follows from item 2. So consider $a_0=1$. Since $a_i \cdot a_{i-1}=0$,
the string $a_N\cdots a_1a_0$ does not contain two consecutive
digits 1, i.e.\ there is a $k\geq 0$ such that $a_N\cdots
a_1a_0=\cdots 0 0 (10)^k1$. One can easily verify that
$11(01)^k\ol{1}\bullet_{-\tau}$ is a $(-\tau)$-representation of
0. Thus for the $(-\tau)$-representation of $y$ we can rewrite
$$
\begin{array}{ccccccccc}
y & = & \cdots & 0 & 0 & (1 & 0)^k & 1 & \bullet_{-\tau}\\
0 & = &        & 1 & 1 & (0 & 1)^k & \ol{1} & \bullet_{-\tau}\\
\hline
y & = & \cdots & 1 & 1 & (1 & 1)^k & 0 & \bullet_{-\tau}
\end{array}
$$
It suffices now to apply item 2.
\end{proof}

In the proof of the following theorem we shall use the
automorphism on the algebraic field $\Q(\tau)$. Recall that $\tau$
is an algebraic number with conjugate $\tau'=-\frac1\tau$. ($\tau$
and $\tau'$ are roots of the polynomial $x^2-x-1$.) Since the
algebraic fields $\Q(\tau)$ and $\Q(\tau')$ coincide, the
isomorphism $\sigma$ defined by~\eqref{eq:isomorphism} is now an
involutive automorphism of the field $\Q(\tau)$, i.e.
$\sigma^2(z)=z$ for all $z\in\Q(\tau)$. We also use the fact known
from~\cite{FruSo} that the set ${\rm Fin}(\tau)$ of numbers with
finite $\tau$-expansions is a ring.

\begin{thm}\label{thm:tau}
$L_\oplus(-\tau)=2=L_\otimes(-\tau)$.
\end{thm}

\begin{proof}
Consider $x,y\in\Z_{-\tau}$, i.e.
$$
x= \sum_{i=0}^k x_{i}(-\tau)^i\,,\qquad y=\sum_{i=0}^l y_{i}(-\tau)^i
$$
with $(-\tau)$-expansions
$$
\begin{array}{rcr}
\langle x\rangle_{-\tau}&=&x_kx_{k-1}\cdots x_0\bullet_{-\tau}\,,\\
\langle y\rangle_{-\tau}&=&y_ly_{l-1}\cdots y_0\bullet_{-\tau}\,.
\end{array}
$$
Applying the automorphism $\sigma$ to $x,y$, we obtain --- by using $-\tau'=\tau^{-1}$ --- that
$$
\sigma(x)= \sum_{i=0}^k x_{i}(-\tau')^i =
\sum_{i=-k}^0x_{-i}\tau^i\,,\qquad \sigma(y)=\sum_{i=0}^l
y_{i}(-\tau')^i=\sum_{i=-l}^0y_{-i}\tau^i\,.
$$
Digit strings
$x_0\bullet_{\tau}x_{1}\cdots x_k$, $y_0\bullet_{\tau}y_{1}\cdots
y_l$ are $\tau$-representations of numbers $\sigma(x)$,
$\sigma(y)$. Since  $x_i,y_i\in\{0,1\}$, we can use Lemma~\ref{celacisla} to derive that
$\sigma(x)$, $\sigma(y)$ belong to ${\rm Fin}(\tau)$. As ${\rm
Fin}(\tau)$ is closed under addition, we have
$\sigma(x)+\sigma(y)=z$ for some $z\in{\rm Fin}(\tau)$. We have
$$
\sigma(x),\sigma(y) < \sum_{i=-\infty}^0\tau^i = \frac{1}{1-\tau^{-1}} = \tau^2\,,
$$
and therefore $z < 2\tau^2 <\tau^4$. This implies that the $\tau$-expansion of $z$ is
of the form
$$
\langle z\rangle_{\tau}=z_3z_2z_1z_0\bullet_{\tau}z_{-1}\cdots z_{-N}
$$
for some integer $N$, i.e.
$$
z=\sum_{i=-N}^{3}z_i\tau^i\,,\qquad z_i\in\{0,1\}\,,\quad
z_i\cdot z_{i-1}=0\,.
$$
Applying the automorphism $\sigma$ to $z$, we obtain
$$
\sigma(z)=\sum_{i=-N}^{3}z_i{\tau'}^i =
\sum_{i=-3}^{N}z_{-i}(-\tau)^i\,.
$$
The string $z_{-N}\cdots z_{0}\bullet_{-\tau} z_1z_2z_3$ is a $(-\tau)$-representation of
$\sigma(z)$. Since $z_i\in\{0,1\}$ and $z_i\cdot z_{i-1}=0$, by item 3.\ of
Lemma~\ref{celacisla}, we get
$$
\sigma(z)=\sigma\big(\sigma(x)+\sigma(y)\big)=x+y\in(-\tau)^{-2}\Z_{-\tau}\,.
$$
By that, we have shown $L_{\oplus}(-\tau)\leq 2$. The opposite inequality is verified by giving the example
$$
1111\bullet_{-\tau}+1111\bullet_{-\tau}=110000\bullet_{-\tau}11\,.
$$

The procedure for proving $L_\otimes(-\tau)\leq 2$ is analogous to
the case of addition. Here $\sigma(x)\cdot\sigma(y)=z$ where $z <
(\tau^2)^2 =\tau^4$, thus we obtain the same number of fractional
digits. In order to prove $L_\otimes(-\tau)\geq 2$, we take the
example
$$
1111\bullet_{-\tau}\times1111\bullet_{-\tau}=11100\bullet_{-\tau}11\,.
$$
\end{proof}

%%%%%%%%%%%%%%%%%%%%%%%%%%%%%%%%%%%%%%%%%%%%%%%%%%%%%%%%%%%%%%%%%%%%%%%%%%
\section{$(-\tau)$-integers}

In the previous section we have focused on the arithmetical properties of the base $-\tau$. Here we show that
the set of non-negative $(-\tau)$-integers coincides with non-negative $\beta$-integers where $\beta=\tau^2$.
%
%Note that this does not necessarily mean that arithmetical operations with numbers written in base $-\tau$ and in base $\tau$
%work in the same way, since the number of digits .....?
%
%We will also comment on the fact that such a relation is exceptional and cannot hold for other quadratic bases.
%
We give the proof by first showing that the distances between  $(-\tau)$-integers take the same values as
the distances between $(\tau^2)$-integers. Then we show that the infinite word coding $(-\tau)$-integers is the fixed point
of the same morphism as for $(\tau^2)$-integers.

First recall the facts about $(\tau^2)$-integers. Realize that
$\beta=\tau^2$ is the greater root of $x^2-3x+1$. The infinite
R\'enyi expansion of $1$ in the base $\tau^2$ is therefore equal
to $d_{\tau^2}^*(1)=21^\omega$. Recall (see~\cite{Thurston}) that the distances between consecutive
$(\tau^2)$-integers take values
$$
\Delta_0=1\quad\text{ and }\quad\Delta_1= \frac1\tau\,.
$$
Writing the distances between the $(\tau^2)$-integers, as they are
ordered, one obtains the bidirectional infinite word
\begin{equation}\label{eq:bidirpositive}
\cdots\Delta_0\Delta_1\Delta_0\Delta_0|\Delta_0\Delta_0\Delta_1\Delta_0\cdots\,,
\end{equation}
where the delimiter $|$ marks the position of 0. Note that the
word is symmetric with respect to 0, since $\Z_{\tau^2}$ is
symmetric with respect to 0 by definition. It is therefore
sufficient to study the one-directional infinite word
$$
u_{\tau^2} = \Delta_0\Delta_0\Delta_1\Delta_0\cdots
$$
coding the distances between non-negative $(\tau^2)$-integers. It
is known~\cite{fabre} that the infinite word $u_{\tau^2}$ is the
fixed point of the following morphism over the alphabet
$\{\Delta_0,\Delta_1\}$,
\begin{equation}\label{eq:subst}
\varphi(\Delta_0)=\Delta_0\Delta_0\Delta_1\,,\qquad\varphi(\Delta_1)=\Delta_0\Delta_1\,.
\end{equation}
Repeated application of the morphism $\varphi$ on the letter
$\Delta_0$ leads to
$$
\Delta_0 \mapsto \Delta_0\Delta_0\Delta_1 \mapsto
\Delta_0\Delta_0\Delta_1\Delta_0\Delta_0\Delta_1\Delta_0\Delta_1
\mapsto \cdots
%\Delta_0\Delta_0\Delta_1\Delta_0\Delta_0\Delta_1\Delta_0\Delta_1
%\Delta_0\Delta_0\Delta_1\Delta_0\Delta_0\Delta_1\Delta_0\Delta_1\Delta_0\Delta_0\Delta_1\Delta_0\Delta_1
$$
where every iteration has the previous iteration as its prefix.
Infinite repetition leads to the word
$u=\lim_{n\to\infty}\varphi^n(\Delta_0) $, where the limit is
taken with respect to the product topology. We have $u= u_{\tau^2}$.

In what follows, we show that the distances between consecutive
$(-\tau)$-integers also take values $\Delta_0=1$ and
$\Delta_1=\tau^{-1}$, and their ordering corresponds to an
infinite word which is a fixed point of the
substitution~\eqref{eq:subst}. For that we use
Lemma~\ref{l:zakazane} which characterizes digit strings
admissible as $(-\tau)$-expansions using forbidden strings. For
$m=n=1$, Lemma~\ref{l:zakazane} states that the digit string
$x_Nx_{N-1}\dots x_1 x_0\,\bullet$, $x_i\in\{0,1\}$, is the
$(-\tau)$-expansion of some $(-\tau)$-integer if and only if $x_Nx_{N-1}\dots x_1 x_0 0^\omega$
does not contain the forbidden string $10^{j}1$ for any $j$ odd,
and it does not end in $0\,1\,0^\omega$.

\begin{lem}
Let $x,y$ be consecutive $(-\tau)$-integers, $x<y$. If the $(-\tau)$-expansion
of $x$ is of the form $\langle
x\rangle_{-\tau}=x_Nx_{N-1}\cdots x_1 0\,\bullet$, then
$y-x=\Delta_0=1$. If $\langle x\rangle_{-\tau}=x_Nx_{N-1}\cdots x_1
1\,\bullet$, then $y-x=\Delta_1=\frac{1}{\tau}$.
\end{lem}

\pfz We distinguish two cases according to the last digit of the
$(-\tau)$-expansion of $x$. We provide the $(-\tau)$-expansion
of $x+\Delta$, where $\Delta=1$ in Case 1. and $\Delta=\tau^{-1}$ in
Case 2. It is then easy to check that no digit string
$z_lz_{l-1}\cdots z_1z_0\,\bullet$ lies between $\langle
x\rangle_{-\tau}$ and $\langle x+\Delta\rangle_{-\tau}$ in
alternate order, and thus no $(-\tau)$-integer lies between $x$
and $x+\Delta$. Therefore $x+\Delta=y$ is the right neighbor of
$x$.

\noindent{\bf Case 1.} Let $\langle
x\rangle_{-\tau}=x_Nx_{N-1}\cdots x_1 0\,\bullet$. We give the
$(-\tau)$-expansion of $x+1$.

  \noindent{\bf Subcase 1.1.} If $x=0$, we add $1=\tau^2-\tau$ to obtain
  $$
  \begin{array}{ccccc@{\ \scriptsize \bullet}}
  \langle x\rangle_{-\tau}&= &0&0 & 0\\
  \langle x+1\rangle_{-\tau}&=&1& 1& 0
  \end{array}
  $$

  \noindent{\bf Subcase 1.2.} If $\langle x\rangle_{-\tau}$ ends in exactly one 0, then we have for $k\geq1$
  $$
  \begin{array}{ccccc@{\ \scriptsize \bullet}}
  \langle x\rangle_{-\tau}&=&\cdots& 1^k & 0\\
  \langle x+1\rangle_{-\tau}&=&\cdots& 1^k& 1
  \end{array}
  $$

  \noindent{\bf Subcase 1.3.} If $\langle x\rangle_{-\tau}$ ends in two 0s, then
  $$
  \begin{array}{ccccccc@{\ \scriptsize \bullet}}
  \langle x\rangle_{-\tau}&=&\cdots& 1 & 1&0&0\\
  \langle x+1\rangle_{-\tau}&=&\cdots& 0& 0&1&1
  \end{array}
  $$

\noindent{\bf Subcase 1.4.} If $\langle x\rangle_{-\tau}$ ends in
an odd number of 0s, $2k+3$, $k\geq0$, then
  $$
  \begin{array}{ccccccccc@{\ \scriptsize \bullet}}
  \langle x\rangle_{-\tau}&=&\cdots& 1 & 1& 0^{2k}&0&0&0\\
  \langle x+1\rangle_{-\tau}&=&\cdots& 1& 1&0^{2k}&1&1&0
  \end{array}
  $$

  \noindent{\bf Subcase 1.5.} If $\langle x\rangle_{-\tau}$ ends in
an even number of 0s, $2k+4$, $k\geq0$, then
  $$
  \begin{array}{cccccccccc@{\ \scriptsize \bullet}}
  \langle x\rangle_{-\tau}&=&\cdots& 1 & 1&0&0^{2k}&0&0&0\\
  \langle x+1\rangle_{-\tau}&=&\cdots& 0& 0&1&0^{2k}&1&1&0
  \end{array}
  $$

  \noindent{\bf Case 2}
Let $\langle x\rangle_{-\tau}=x_Nx_{N-1}\cdots x_1 1\,\bullet$.
Note that $x_1$ is necessarily equal to $1$. We give the
$(-\tau)$-expansion of $x+\tau^{-1}$.

  \noindent{\bf Subcase 2.1.} If $x=-\tau^{-1}$, then
  $$
  \begin{array}{cccc@{\ \scriptsize \bullet}}
  \langle x\rangle_{-\tau}&=& 1 & 1\\
  \langle x+\tau^{-1}\rangle_{-\tau}&=& 0&0
  \end{array}
  $$

  \noindent{\bf Subcase 2.2.} $k\geq0$
  $$
  \begin{array}{cccccccc@{\ \scriptsize \bullet}}
  \langle x\rangle_{-\tau}&=&\cdots& 1 & 1&0^{2k}&1&1\\
  \langle x+\tau^{-1}\rangle_{-\tau}&=&\cdots& 1&1&0^{2k}&0&0
  \end{array}
  $$

  \noindent{\bf Subcase 2.3.} $k\geq0$
  $$
  \begin{array}{ccccccccc@{\ \scriptsize \bullet}}
  \langle x\rangle_{-\tau}&=&\cdots& 0&0 & 1&0^{2k}&1&1\\
  \langle x+\tau^{-1}\rangle_{-\tau}&=&\cdots& 1&1&0&0^{2k}&0&0
  \end{array}
  $$
\pfk

\begin{lem}
Let $x,y\in\Z_{-\tau}$ be consecutive $(-\tau)$-integers, $x<y$.
If $y-x=1$, then
$$
[\tau^2x,\tau^2y] \cap \Z_{-\tau} = \{\tau^2x,\tau^2x+1,
\tau^2x+1+1,\tau^2x+1+1+\tau^{-1} =\tau^2y \}\,.
$$
If $y-x=\tau^{-1}$, then
$$
[\tau^2x,\tau^2y] \cap \Z_{-\tau} = \{\tau^2x,\tau^2x+1,
\tau^2x+1+\tau^{-1} =\tau^2y \}\,.
$$
\end{lem}

\pfz Let first $y-x=1$.  The $(-\tau)$-expansion of $-\tau y$
ends obviously in $0$, and thus the right neighbor of $-\tau y$
among $(-\tau)$-integers is $-\tau y +1$. Since $-\tau x = -\tau y
+\tau = (-\tau y + 1) + \tau^{-1}$ is also a $(-\tau)$-integer, we
have
\begin{equation}\label{eq:8}
[-\tau y,-\tau x] \cap \Z_{-\tau} = \{-\tau y,-\tau y+1, -\tau y +
1 + \tau^{-1}= -\tau x \}\,.
\end{equation}

Let now $y-x=\tau^{-1}$.  The $(-\beta)$-expansion of $-\tau y$
ends again in $0$, and thus the right neighbor of $-\tau y$ among
$(-\tau)$-integers is $-\tau y +1 = -\tau x$. Therefore
\begin{equation}\label{eq:9}
[-\tau y,-\tau x] \cap \Z_{-\tau} = \{-\tau y,-\tau y+1= -\tau x
\}\,.
\end{equation}
Applying the rules~\eqref{eq:8} and~\eqref{eq:9} twice, we obtain
the result.
 \pfk

Writing the distances between consecutive $(-\tau)$-integers by
symbols $\Delta_0,\Delta_1$, the above lemma states that the
infinite word
$$
u_{-\tau} = \cdots
\Delta_1\Delta_0\Delta_1|\Delta_0\Delta_0\Delta_1\Delta_0\cdots
$$
coding $(-\tau)$-integers is invariant under the
morphism~\eqref{eq:subst}, i.e.\
\begin{equation}\label{eq:bidir}
u_{-\tau} = \cdots \Delta_0\Delta_1|\Delta_0\Delta_0\Delta_1\cdots
= \cdots
\varphi(\Delta_0)\varphi(\Delta_1)|\varphi(\Delta_0)\varphi(\Delta_0)\varphi(\Delta_1)\cdots
\end{equation}
We have thus shown the following theorem.

\begin{thm}
$\Z_{-\tau}\cap [0,+\infty) = \Z_{\tau^2} \cap [0,+\infty)$.
\end{thm}

Note that the theorem connects only the non-negative part of $(-\tau)$- and $(\tau^2)$-integers. This is because whereas
$\Z_{\tau^2}$ is in some sense `artificially' defined on the negative half-line,  negative $(-\tau)$-integers are defined naturally.
The bidirectional infinite word $u_{-\tau}$ of~\eqref{eq:bidir} can be seen by the bidirectional limit $\lim_{n\to\infty}\varphi^n(1)|\varphi^n(0)$.
The same is not true for the bidirectional word~\eqref{eq:bidirpositive} coding the $(\tau^2)$-integers over all real line.
Such phenomena could be, for example, used to solve problems mentioned in~\cite{BuFrGaKr}.

%%%%%%%%%%%%%%%%%%%%%%%%%%%%%%%%%%%%%%%%%%%%%%%%%%%%%%%%%%%%%%%%%%%%%%%%%%
\section{Conclusions and open problems}

The present paper answers several questions raised in~\cite{MaPeVa} about arithmetics in numeration systems with negative base $-\beta$ where
$\beta$ is a quadratic Pisot number. Nevertheless, other problems about arithmetical properties of such systems remain unsolved, such as efficient algorithms for performing addition and multiplication, description of numbers with purely periodic $(-\beta)$-expansion, etc.

There are also other aspects of numeration in non-standard systems that deserve to be explored.
Properties, known for R\'enyi numeration with positive base, likely to hold also for the case of negative base, are for example the description of
the distribution of $(-\beta)$-integers, given for positive base in~\cite{asymptotika}, or the connection of $(-\beta)$-numeration to substitution dynamical systems established for $\beta$-integers in~\cite{fabre}. Many such properties are studied in~\cite{KalleSteiner} for number systems
with positive base which are generalizations of the R\'enyi numeration.

It may, however, happen that the analogues for negative base will not be straightforward. Although numerical experiments suggest that the analogues of certain properties are valid, the classical methods for proofs fail. An example of such a situation is the study of existence of morphisms
generating the set of $(-\beta)$-integers performed by different methods in~\cite{ADMP} and~\cite{SteinerSubstituce}.
%For example, Steiner~\cite{SteinerSubstituce} has shown existence of a morphism generating the set of $(-\beta)$-integers,
%however, without providing an explicit prescription for such a morphism. The same question is studied in~\cite{ADMP}
%where explicit prescription is given but not for all cases of relevant bases $-\beta$.

The relation of numeration with positive base and negative base should be further studied. In here, we show a -- in the general context rather surprising -- coincidence between integers in the numeration systems with base $-\tau$, where $\tau$ is the golden ratio, and integers in the system with positive base $\tau^2$. Such a simple relation, however, cannot be expected even in case of other quadratic bases. For, as shown in~\cite{ADMP}, the distances between consecutive $(-\beta)$-integers need not be smaller than 1, and it is
even not obvious whether they are in general bounded.

% as an example, consider negative base $-\beta$, where $\beta$ is smaller than the golden ratio.
%Here, the set of finite $(-\beta)$-expansions is trivial --- a phenomena unknown for positive base numeration.

%%%%%%%%%%%%%%%%%%%%%%%%%%%%%%%%%%%%%%%%%%%%%%%%%%%%%%%%%%%%%%%%%%%%%%%%%%
\section*{Acknowledgement}

We acknowledge financial support by the Czech Science Foundation
grant 201/09/0584 and by the grants MSM6840770039 and LC06002 of
the Ministry of Education, Youth, and Sports of the Czech
Republic. The work was also partially supported by the CTU student grant SGS10/085/OHK4/1T/14.

%%%%%%%%%%%%%%%%%%%%%%%%%%%%%%%%%%%%%%%%%%%%%%%%%%%%%%%%%%%%%%%%%%%%%%%%%%

%%%%%%%%%%%%%%%%%%%%%%%%%%%%%%%%%%%%%%%%%%%%%%%%%%%%%%%%%%%%%%%%%%%%%%%%%%

\end{document}